\documentclass[11pt]{jdg-p.cls}

\usepackage{amsmath, amsthm}
\usepackage{amssymb}
\usepackage{amsfonts}
\usepackage{latexsym}
\usepackage{mathrsfs}
\usepackage{bm}
\usepackage{graphicx}

\newcommand{\grad}{\nabla}

\renewcommand{\L}{\mathcal{L}}

\DeclareMathOperator{\re}{Re}
\renewcommand{\Re}{\re}

\renewcommand{\tilde}[1]{\widetilde{#1}}

\newcommand{\tec}{Teichm\"uller }

\renewcommand{\leq}{\leqslant}
\renewcommand{\geq}{\geqslant}

\DeclareMathOperator{\Vol}{Vol}
\DeclareMathOperator{\dVol}{dVol}

\DeclareMathOperator{\dist}{dist}
\DeclareMathOperator{\inj}{inj}

\DeclareMathOperator{\Teich(S_g)}{Teich(S_g)}
\DeclareMathOperator{\Te}{Teich}
\DeclareMathOperator{\Ric}{Ric}
\DeclareMathOperator{\sign}{sign}
\DeclareMathOperator{\QD(X)}{QD(X)}
\DeclareMathOperator{\HBD(X)}{HBD(X)}

\DeclareMathOperator{\Sca}{Sca}
\DeclareMathOperator{\asydim}{asydim}
\DeclareMathOperator{\Diff}{Diff}

\newcommand{\Mod}{\mbox{\rm Mod}}

\newcommand{\param}{{\mathchoice{\mkern1mu\mbox{\raise2.2pt\hbox{$\centerdot$}}\mkern1mu}{\mkern1mu\mbox{\raise2.2pt\hbox{$\centerdot$}}\mkern1mu}{\mkern1.5mu\centerdot\mkern1.5mu}{\mkern1.5mu\centerdot\mkern1.5mu}}}

\numberwithin{equation}{section}

\theoremstyle{plain}
\newtheorem{theorem}{Theorem}[section]

\newtheorem{lemma}[theorem]{Lemma}
\newtheorem{proposition}[theorem]{Proposition}

\newtheorem*{acknowledgement}{Acknowledgement}

\theoremstyle{definition}
\newtheorem{definition}[theorem]{Definition}
\newtheorem{question}{Question}

\newtheorem{remark}[theorem]{Remark}
\theoremstyle{definition}
\newtheorem*{remarksenv}{Remarks}

\begin{document}

\date{}
\title[Scalar Curvature]
{On positive scalar curvature and moduli of curves}

\author{Kefeng Liu and Yunhui Wu}



\address{Department of Mathematics\\
        University of California, Los Angeles\\
        Los Angeles, CA 90095-1555\\}
\email{liu@math.ucla.edu\\}

\address{Department of Mathematics\\
       Rice University\\
     Houston, Texas, 77005-1892\\}
\email{yw22@rice.edu}


\begin{abstract}
In this article we first show that any finite cover of the moduli space of closed Riemann surfaces of genus $g$ with $g\geq 2$ does not admit any Riemannian metric $ds^2$ of nonnegative scalar curvature such that $ds^2 \succ ds_{T}^2$ where $ds_{T}^2$ is the Teichm\"uller metric.

Our second result is the proof that any cover $M$ of the moduli space $\mathbb{M}_{g}$ of a closed Riemann surface $S_{g}$ does not admit any complete Riemannian metric of uniformly positive scalar curvature in the quasi-isometry class of the Teichm\"uller metric, which implies a conjecture of Farb-Weinberger in \cite{Farb-prob}.
\end{abstract}


\maketitle
\section{Introduction}
Many aspects of positive scalar curvatures on Riemannian manifolds have been well understood since the fundamental works of Schoen-Yau \cite{SchoenYau79-1, SchoenYau79-2} and Gromov-Lawson \cite{GL80,GL83}. Important generalizations have been developed by Roe \cite{Roe93}, Yu \cite{Yu98} and many others. The main object of this paper is to study obstructions to the existence of positive scalar curvature metric on the moduli spaces of Riemann surfaces.

Let $S_g$ be a closed Riemann surface of genus $g$ with $g\geq 2$, $\Mod(S_g)$ be the mapping class group and $\Teich(S_g)$ be the Teichm\"uller space of $S_g$. Topologically $\Teich(S_g)$  is a manifold of real dimension $6g-6$, which carries various $\Mod(S_g)$-invariant metrics which descend to metrics on the moduli space $\mathbb{M}_g$ of $S_g$ with respective properties. For examples, the famous Weil-Petersson metric and Teichm\"uller metric.  The Teichm\"uller metric $ds_{T}^2$ is not Riemannian but is a complete Finsler metric. It was shown in \cite{Masur75} that $ds_T^2$ is not nonpositively curved in the metric sense by showing that there exists two different geodesic rays starting at the same point such that they have bounded Hausdorff distance. Furthermore, Masur and Wolf in \cite{MW95} showed that $(\Teich(S_g),ds_T^2)$ is not Gromov-hyperbolic. The Weil-Petersson metric $ds_{WP}^2$ is K\"ahler \cite{Ahlfors61}, incomplete \cite{Chu76,  Wolpert75}, geodesically convex \cite{Wolpert87} and has negative sectional curvature \cite{Wolpert86, Tromba86}. Since $g\geq 2$, the works in \cite{Wolpert03, Yamada04} tell that $\Teich(S_g)$ is not Gromov-hyperbolic. Let $\Te(S_{g,n}$ be the Teichm\"uller space of surfaces of genus $g$ with $n$ punctures. Brock and Farb in \cite{BF02} showed that the space $(\Te(S_{g,n}),ds_{WP}^2)$ is Gromov-hyperbolic if and only if $3g+n\leq 5$.

There are also some other important metrics on $\mathbb{M}_g$. For examples, the asymptotic Poincar\'e metric, the induced Bergman metric, the K\"ahler-Einstein metric, the McMullen metric, the Ricci metric, and the perturbed Ricci metric are all complete and K\"ahler metrics. The Kobayashi metric and the Carathe\'odory metric are complete and Finsler metrics. In \cite{LSY04, LSY05, McMullen00}, the authors showed that all these metrics are bi-Lipschitz (or equivalent) to the Teichm\"uller metric. And the Weil-Petersson metric plays important roles in their proofs.

The perturbed Ricci metric \cite{LSY04, LSY05} has pinched negative Ricci curvature. In particular it also has negative scalar curvature. The McMullen metric \cite{McMullen00} has negative scalar curvature at certain points since the metric, restricted on certain thick part of the moduli space, is the Weil-Petersson metric. However, Farb and Weinberger in \cite{FW-scalar} showed that any finite cover $M$ of the moduli space $\mathbb{M}_g$ $(g\geq 2)$ admits a complete finite-volume Riemannian metric of (uniformly bounded) positive scalar curvature, which is analogous to Block-Weinberger's result in \cite{BW99} on certain locally symmetric arithmetic manifolds.

The motivation of this paper is a result of Gromov-Lawson in  \cite{GL83} which says that \textsl{given a complete Riemannian manifold $(X,ds_1^2)$ of nonpositive sectional curvature, then $X$ can not admit any Riemannian metric $ds_2^2$ on $X$ with $ds_2^2 \succ ds_1^2$ such that $(X,ds_2^2)$ has positive scalar curvature}. (One can also see Theorem 1.1 in \cite{Roe93}) where $ds_2^2 \succ ds_1^2$ means that $ds_2^2 \geq k \cdot ds_1^2$ for some constant $k>0$. One immediate application is that the torus $\mathbb{T}^n$ ($n\ge 2$) can not carry a complete Riemannian metric of positive scalar curvature, which answers a question of Geroch. For low dimensions $n\leq 7$, this was first settled in a series of papers by R. Schoen and S. T. Yau in \cite{SchoenYau79-1} and \cite{SchoenYau79-2}.  

Our first result in this paper is the following theorem.
\begin{theorem}\label{mt-1}
Let $S_{g}$ be a closed Riemann surface of genus $g$ with $g\geq 2$ and $M$ be a finite cover of the moduli space $\mathbb{M}_{g}$ of $S_{g}$. Then for any Riemannian metric $ds^2$ on $M$ with $ds^2 \succ ds_T^2$ where $ds_T^2$ is the \tec metric, 
$$\inf_{p\in (M,ds^2)}\Sca(p)<0.$$
\end{theorem}

E. Leuzinger in \cite{leuz2010} showed that \textsl{any finite cover of the moduli space $\mathbb{M}_{g}$ of $S_{g}$ does not admit any Riemannian metric of uniformly positive scalar curvature such that the metric is bi-lipschitz to the \tec metric.} Theorem \ref{mt-1} generalizes his result and the method in this article is completely different from the one of E. Leuzinger in \cite{leuz2010}. The second author in \cite{Wu15} applied the negativity of the Ricci curvature of the perturbed Ricci metric \cite{LSY04, LSY05} to show that  \textsl{ any finite cover of the moduli space $\mathbb{M}_{g}$ does not admit any complete finite-volume Hermitian metric of nonnegative scalar curvature}. Moreover, he also showed that \textsl{ the total scalar curvature of any almost Hermitian metric, which is bi-Lipschitz to the Teichm\"uller metric, is negative provided the scalar curvature is bounded from below.} However, the method in \cite{Wu15} highly depends on the canonical complex structure on $\mathbb{M}_{g}$, which fails in the setting of Riemannian metrics. 

Theorem \ref{mt-1} applies to the metric $ds^2=ds^2+ds_a^2$ where $ds^2$ is either the McMullen metric or the perturbed Ricci metric and $ds_a^2$ is only Riemannian and not Hermitian. We remark here that there is no finite-volume condition for Theorem \ref{mt-1}. As stated above the Teichm\"uller metric is not nonpositively curved. So the argument in \cite{GL83} can not lead to Theorem \ref{mt-1}. We are going to use some recent developments in \cite{McMullen00, LSY04, LSY05, BBF14} on the geometry of Teichm\"uller space as bridges to prove Theorem \ref{mt-1}. 

\begin{remark}
(1). Following totally the same arguments in this article, one can deduce that Theorem \ref{mt-1} is still true for the Teichm\"uller space of noncompact surface $S_{g,n}$ of genus $g$ with $n$ punctures if $3g+n\geq 5$. Note that Theorem \ref{Kazdan-4-1} and \ref{deform} require that the dimension of the space is greater than or equal to $3$. 

(2). For the cases $(g,n)=\ (1,1) \ \textit{or}\ (0,4)$, it is not hard to see that Theorem \ref{mt-1} still holds without the assumptions on the Teichm\"uller metric, since the scalar curvature is exact sectional curvature for these cases. More precisely, Theorem \ref{mt-1} directly follows from the fact that the mapping class group contains free subgroups of rank $\geq 2$. 

(3). The existence theorem of Farb-Weinberger in \cite{FW-scalar}, which one can also see Theorem 4.5 in \cite{Farb-prob}, tells that Theorem \ref{mt-1} does not hold anymore without  the assumptions on the Teichm\"uller metric if $3g+n \geq 6$. It is \textsl{interesting} to know whether Theorem \ref{mt-1} is still true without the assumption on the Teichm\"uller metric when $3g+n=5$. 

\end{remark}

As stated above Farb and Weinberger in \cite{FW-scalar} proved the existence of complete Riemannian metrics of uniformly positive scalar curvature on the moduli space. Actually they also showed that these metrics are not quasi-isometric to the Teichm\"uller metric. Motivated by Chang's result in \cite{chang01} on certain locally symmetric spaces, they conjecture in \cite{Farb-prob} (see Conjecture 4.6 in \cite{Farb-prob}) that \textsl{any finite cover $M$ of the moduli space $\mathbb{M}_{g}$ of $S_{g}$ does not admit a finite volume Riemannian metric of (uniformly bounded) positive scalar curvature in the quasi-isometry class of the Teichm\"uller metric.} Instead of using Theorem \ref{gromovlawson} of Gromov-Lawson in the proof of Theorem \ref{mt-1}, we apply a theorem of Yu in \cite{Yu98} to give a short proof of the following result, which in particular gives a proof of this conjecture of Farb-Weinberger.
\begin{theorem}\label{mt-2}
Let $S_{g}$ be a closed surface of genus $g$ with $g\geq 2$. Then any cover $M$ of the moduli space $\mathbb{M}_{g}$ of $S_{g}$ does not admit a complete Riemannian metric of uniformly positive scalar curvature in the quasi-isometry class of the Teichm\"uller metric.
\end{theorem}

There are no conditions on finite cover and finite volume in Theorem \ref{mt-2}, compared to the Farb-Weinberger conjecture. We remark here that Farb and Weinberger have a different approach to their conjecture by using methods from Chang's thesis \cite{chang01} together with a theorem of Farb-Masur in \cite{FM2010} on the asymptotic cone of the moduli space, which is different from the method in this article. We thank him for sharing their information.  

\subsection{Plan of the paper.} In Section \ref{np} we give some necessary preliminaries and notations for surface theory. In Section \ref{universal} we review some recent developments on the geometry of Teichm\"uller space which will be served as bridges to prove Theorem \ref{mt-1}. In Section \ref{dtpsc} we will show that any complete Riemannian metric on the moduli space of surfaces with nonnegative scalar curvature can be deformed an equivalent Riemannian metric of positive scalar curvature. Theorem \ref{mt-1} will be proved in Section \ref{proof of 1}. In Section \ref{proof of 2} we will disucss Theorem \ref{mt-2} in details. A related open problem will be discussed in Section \ref{question}.

\begin{acknowledgement}
The authors would like to thank E. Leuzinger, G. Yu and S. T. Yau for their interests. The first author is supported by an NSF grant. The second author is grateful to J. Brock and M. Wolf for their consistent encouragement and help. He also would like to acknowledge support from U.S. National Science Foundation grants DMS 1107452, 1107263, 1107367 ``RNMS: Geometric structures And Representation varieties"(the GEAR Network).
\end{acknowledgement}

\section{Notations and Preliminaries}\label{np} 

\subsection{Surfaces}
Let $S_g$ be a closed Riemann surface of genus $g$ with $g\geq 2$, and $\textsl{M}_{-1}$ denote the space of Riemannian metrics on $S_g$ with constant curvature $-1$, and $X=(S_g,\sigma|dz|^2)$ be an element in $\textsl{M}_{-1}$. The group $\Diff_0$ of diffeomorphisms of $S_g$ isotopic to the identity, acts  on $\textsl{M}_{-1}$ by pull-backs. The Teichm\"uller space $\Teich(S_g)$  of $S_g$  is defined by the quotient space
\begin{equation}
\nonumber \Teich(S_g)=M_{-1}/\Diff_0.  
\end{equation}

Let $\Diff_+$ be the group of orientation-preserving diffeomorphisms of $S_g$. The mapping class group $\Mod(S_g)$ is defined as
\begin{equation}
\nonumber \Mod(S_g)=\Diff_+/\Diff_0.  
\end{equation}
The moduli space $\mathbb{M}_g$ of $S_g$ is defined by the quotient space
\begin{equation}
\nonumber \mathbb{M}_g=\Teich(S_g)/\Mod(S_g).  
\end{equation}

The Teichm\"uller space has a natural complex structure, and its holomorphic cotangent space $T_{X}^{*}\Teich(S_g)$ is identified with the \textsl{quadratic differentials} 
$$\QD(X)=\{\phi(z)dz^2\}$$ 
while its holomorphic tangent space is identified with the \textsl{harmonic Beltrami differentials} 
$$\HBD(X)=\{\frac{\overline{\phi(z)}}{\sigma(z)}\frac{d\overline{z}}{dz}\}.$$ 

\subsection{Mapping class group} The mapping class group $\Mod(S_g)$ is a finitely generated discrete group which acts properly on the Teichm\"uller space.  One special set of generators  of $\Mod(S_g)$  is the Dehn-twists along simple closed curves, which play an important role in studying $\Mod(S_g)$. Let $\alpha$ be a nontrivial simple closed curve on $S_g$ and $\tau_{\alpha}$ be the Dehn-twist along $\alpha$. The following lemma will be applied later.

\begin{lemma}\label{dt-f}
Let $\alpha, \beta$ be two simple closed curves on $M$ such that the geometric intersection points $i(\alpha, \beta)\geq 2$. Then, for any $n,m \in \mathbb{Z}^{+}$, the group 
$<\tau_{\alpha}^n, \tau_{\beta}^m>$ is a free group of rank $2$. 
\end{lemma}  

\begin{proof}
One can check chapter 3 in \cite{FM-mcg} for details.
\end{proof}

\subsection{The Teichm\"uller metric and Weil-Petersson metric}
Recall that the \textsl{Teichm\"uller metric} $ds_T^2$ on $\Teich(S_g)$ is defined as
$$||\frac{\overline{\phi(z)}}{\sigma(z)}||_{ds_T^2}:=\sup_{\psi dz^2 \in \QD(X), \ \ \int_{X}|\psi|=1} \Re  \int_X {\frac{\overline{\phi(z)}}{\sigma(z)} \cdot \psi(z)}\frac{dz\wedge d\overline{z}}{-2\textbf{i}}.$$

The induced path metric of the above metric, denoted by $\dist_{T}$, on $\Teich(S_g)$ can also be characterized as follows; let $p_1, p_2\in \Teich(S_g)$, then 
$$\dist_{T}(p_1,p_2)=\frac{1}{2}\log K$$
where $K\geq 1$ is the least number such that there is a $K$-quasiconformal mapping between the hyperbolic surfaces $p_1$ and $p_2$. The Teichm\"uller metric is not Riemannian but Finsler. The following fundamental theorem on the Teichm\"uller metric will be used later.

\begin{theorem}[Teichm\"uller]\label{teich}
(1). The Teichm\"uller space $(\Teich(S_g), ds_T^2)$ is complete.

(2). The Teichm\"uller space $(\Teich(S_g), ds_T^2)$ is uniquely geodesical, i.e., for any two points $p_1, p_2\in \Teich(S_g)$ there exists a unique geodesic $c:[0,1]\rightarrow (\Teich(S_g), ds_T^2)$ such that $c(0)=p_1$ and $c(1)=p_2$.
\end{theorem}  

A direct corollary is 

\begin{proposition}\label{contract}
Any geodesic ball of finite radius in $(\Teich(S_g), ds_T^2)$ is contractible.
\end{proposition}
\begin{proof}
Let $p \in (\Teich(S_g), ds_T^2)$ and $r>0$. Consider the geodesic ball 
$$B(p;r):=\{q; \ \dist_T(p,q)\leq r\} \subset (\Teich(S_g), ds_T^2).$$

For any $z\in B(p;r)$, by Theorem \ref{teich} we know that there exists a unique geodesic $c_z:[0, \dist_T(p,z)]\rightarrow B(p;r)$ such that $c_z(0)=p$ and $c_z(\dist_T(p,z))=z$. Here we use the arc-length parameter for $c$. Then we consider the following map
\begin{eqnarray*}
H: B(p;r)\times [0,1]&\rightarrow& B(p;r)\\
(z,t) &\mapsto& c_z(t\cdot \dist_T(p,z)).
\end{eqnarray*}
Theorem \ref{teich} tells us that $H$ is well-defined and continuous. 

It is clear that 
$$H(z,0)=p  \quad \mathrm{and} \quad  H(z,1)=z \quad \forall z\in B(p;r).$$
That is, $B(p,r)$ is contractible.
\end{proof}
For more details on Teichm\"uller geometry, one can refer to the book \cite{IT92} and the recent survey \cite{Masur2010} for more details.\\

The \textit{Weil-Petersson metric} $ds_{WP}^2$ is the Hermitian
metric on $T_{g}$ arising from the the \textsl{Petersson scalar  product}
\begin{equation}
 <\varphi,\psi>_{ds_{WP}^2}= \int_X \frac{\varphi(z) \cdot \overline{\psi(z)}}{\sigma(z)}\frac{dz\wedge d\overline{z}}{-2\textbf{i}} \nonumber
\end{equation}
via duality.
The Weil-Petersson metric is K\"ahler (\cite{Ahlfors61}), incomplete (\cite{Chu76, Wolpert75}) and has negative sectional curvature (\cite{Wolpert86, Tromba86}). One can refer to Wolpert's recent book \cite{Wolpertbook} for the progress on the study of the Weil-Petersson metric. 

Both the Teichm\"uller metric and the Weil-Petersson metric are $\Mod(S_g)$-invariant.\\

Let $ds_1^2$ and $ds_2^2$ be any two Riemannian metrics on $\Teich(S_g)$. If there exists a constant $k>0$ such that 
$$ds_1^2 \geq k \cdot ds_2^2,$$
then we write 
$$ds_1^2 \succ ds_2^2.$$ 

We call the two metrics $ds_1^2$ and $ds_2^2$ are $\textit{bi-Lipschitz}$ (or $\textit{equivalent}$) if 
$$ds_1^2 \succ ds_2^2 \quad \textit{and} \quad ds_2^2 \succ ds_1^2$$  
which is denoted by 
$$ds_1^2 \asymp ds_2^2.$$

It is not hard to see that the Cauchy-Schwarz inequality and the Gauss-Bonnet formula gives that 
$$ds_{T}^2 \succ ds_{WP}^2.$$ 
However, since the Weil-Petersson metric is incomplete and the Teichm\"uller metric is complete,  we have 
$$ds_{WP}^2\nsucc ds_{T}^2.$$  

\section{Universal properties of Riemannian metrics  equivalent to $ds_T^2$}\label{universal}

It is shown in \cite{LSY04, LSY05, McMullen00} that the asymptotic Poincar\'e metric, the induced Bergman metric, K\"ahler-Einstein metric, the McMullen metric, the Ricci metric, and the perturbed Ricci metric are all K\"ahler and equivalent to the Teichm\"uller metric. Actually for any metric $ds^2$ in the convex hull of all these metrics we have $ds^2 \asymp ds_{T}^2$. We are going to apply one of these metrics as bridges to prove Theorem \ref{mt-1}. We remark here that certain universal properties of  the six metrics above are enough in this article. We do not need certain special property of some metric in the list above.

 \subsection{K\"ahler metrics on $\mathbb{M}_g$} In this subsection we briefly review some properties of the following two K\"ahler metrics $\mathbb{M}_g$: the Ricci metric and the perturbed Ricci metric. They will be applied to prove Theorem \ref{mt-1}.
\subsubsection{The Ricci metric and the perturbed Ricci metric}

In \cite {Tromba86, Wolpert86} it is shown that the Weil-Petersson metric has negative sectional curvature. The negative Ricci curvature tensor defines a new metric $ds_{\tau}^2$ on $\mathbb{M}_g$, which is called the $\textit{Ricci metric}$. Trapani in \cite{Trapani92} proved $ds_{\tau}^2$ is a complete K\"ahler metric. In \cite{LSY04} Liu-Sun-Yau perturbed the Ricci metric with the Weil-Petersson metric to give new metrics on $\mathbb{M}_g$ which are called the \textsl{perturbed Ricci metrics}, denoted by $ds_{LSY}^2$. More precisely, let $\omega_{\tau}$ be the K\"ahler form of the Ricci metric, for any constant $C>0$, the K\"ahler form of the perturbed Ricci metric is

$$\omega_{LSY}=\omega_{\tau}+C\cdot \omega_{WP}.$$

Motivated by the results of McMullen in \cite{McMullen00},  K. Liu, X. Sun and S. T. Yau  in \cite{LSY04} showed that

\begin{theorem}[Liu-Sun-Yau]\label{lsy-04}
On the moduli space $\mathbb{M}_g$, both $(\mathbb{M}_g,ds_{\tau}^2)$ and $(\mathbb{M}_g,ds_{LSY}^2)$ satisfy

(1). They have bounded sectional curvatures and finite volumes.   

(2). $ ds_\tau^2 \asymp ds_{LSY}^2 \asymp ds_{T}^2 .$

(3). There exists a constant $\epsilon_0>0$ such that the injectivity radii of the universal covers satisfy that 
$$\inj(\Teich(S_g),ds_\tau^2)\geq \epsilon_0>0$$ 
and 
$$\inj(\Teich(S_g),ds_{LSY}^2)\geq \epsilon_0>0.$$
\end{theorem}

\subsection{Asymptotic dimension}
Gromov in \cite{Gromov93} introduced the notion of \textsl{asymptotic dimension} as a large-scale analog of the covering dimension. 
More precisely, a metric space $X$ has asymptotic dimension $\asydim(X)\leq n$ if for every $R>0$ there is a cover of $X$ by uniformly bounded sets such that every metric $R$-ball intersects at most $n+1$ of sets in the cover. One can refer to Theorem 19 in  \cite{BD08} for some other equivalent definitions of the asymptotic dimension. By using Minsky's product theorem in \cite{Minsky96} for the thin part of the Teichm\"uller space $(\Teich(S_g), ds_T^2)$, recently M. Bestvina, K. Bromberg and K. Fujiwara in \cite{BBF14} proved the following result which is crucial for this paper.

\begin{theorem}[Bestvina-Bromberg-Fujiwara]\label{asyofteich}
Let $S_g$ be a closed surface of genus $g$ with $g\geq 1$. Then the Teichm\"uller space, endowed with the Teichm\"uller metric, satisfies
$$\asydim((\Teich(S_g), ds_T^2))<\infty.$$
\end{theorem}

From the definition of the asymptotic dimension it is not hard to see that the asymptotic dimension is a quai-isometric invariance. For more details, one can see the remark on page 21 of \cite{Gromov93} or Proposition 22 in \cite{BD08}. 
\begin{theorem}\label{finite}
Let $ds^2$ be a Riemannian metric on $\Teich(S_g)$ with $ds^2 \asymp ds_{T}^2$. Then,
\begin{equation}\label{finite-1}
\asydim((\Teich(S_g),ds^2))<\infty.
\end{equation}
In particular, for the perturbed Ricci metric $d_{LSY}^2$ we have 
\begin{eqnarray}\label{finite-2}
\asydim((\Teich(S_g),ds_{LSY}^2))<\infty.
\end{eqnarray}
\end{theorem}

\begin{proof}
Since the asymptotic dimension is a quasi-isometry invariant of $\Teich(S_g)$, it is clear that inequality (\ref{finite-1}) follows from Theorem \ref{asyofteich}, and inequality (\ref{finite-2}) follows from Part (2) of Theorem \ref{lsy-04} and inequality (\ref{finite-1}).
\end{proof}

\section{Deformation to Positive Scalar curvature}\label{dtpsc}
As stated in the introduction Farb and Weinberger in \cite{FW-scalar} showed that the set of complete Riemannian metrics of positive scalar curvatures on the moduli space $\mathbb{M}_g$ is not empty. In this section we will show that any complete Riemannian metric of nonnegative scalar curvature on a manifold which finitely covers $\mathbb{M}_g$ can be deformed to a new complete Riemannian metric of positive scalar curvature which is equivalent to the ambient metric. This will be applied to prove Theorem \ref{mt-1}. 

In \cite{Kazdan82} Kazdan showed that any Riemannian metric of zero scalar curvature on a manifold, whose dimension is greater than or equal to $3$, can be deformed to a new metric of positive scalar curvature which is equivalent to the ambient metric provided that the ambient metric is not Ricci flat. Actually his method also works when the scalar curvature is nonnegative. This argument will be used in this section to prove Theorem \ref{deform}. One can see \cite{Kazdan82} for more details.

Let $M$ be a finite cover of the moduli space $\mathbb{M}_g$ and $ds^2$ be a complete Riemannian metric on $M$. Since $M$ may be an orbiford, the Riemannian metric $ds^2$ on $M$ means a Riemannian metric on the Teichm\"uller space $\Teich(S_g)$ on which the orbiford fundamental group $\tilde{\pi_{1}}(M)$ acts on $(\Teich(S_g), ds^2)$ by isometries. It is known that the mapping class group $\Mod(S_g)$ contains torsion-free subgroups of finite indices (see \cite{FM-mcg}). We can pass to a finite cover $\overline{M}$ of $M$ such that $\overline{M}$ is a manifold. It is clear that the fundamental group $\pi_1(\overline{M})$ is a torsion-free subgroup of $\Mod(S_g)$ of finite index. 

In this section we will prove the following result.
\begin{theorem}\label{deform}
Let $S_{g}$ be a closed surface of genus $g$ with $g\geq 2$ and $M$ be a finite cover of the moduli space $\mathbb{M}_{g}$ of $S_{g}$ such that $M$ is a manifold. Then for any complete Riemannian metric $ds^2$ of nonnegative scalar curvature on $M$, there exists a new metric $ds_1^2$ on $M$ such that

(1). The scalar curvature $\Sca_{(M,ds_1^2)}>0$ on $(M,ds_1^2)$.

(2). $ds_1^2 \asymp ds^2.$
\end{theorem}

Before we go to prove the theorem above. First let us provide the following fact in the moduli space.
\begin{lemma}\label{norf}
Let $S_{g}$ be a closed surface of genus $g$  with $g\geq 2$ and $M$ be a finite cover of the moduli space $\mathbb{M}_{g}$ of $S_{g}$. Then, for any complete Riemannian metric $ds^2$ on $M$, there exists a point $p_0 \in M$ such that the Ricci tensor at $p_0$ satisfies
$$\Ric_{(M,ds^2)}(p_0)\neq 0.$$
\end{lemma}

\begin{proof}
The following argument is standard. One may see \cite{Gromovbook} for more applications of this argument. 

We argue it by contradiction. Suppose it is not. That is, there exists a complete Riemannian metric $ds^2$ on $M$ such that for all $p\in M$ the Ricci tensor 
$$\Ric_{(M,ds^2)}(p)=0.$$ 

As described above, if necessary we pass to a finite cover $\overline{M}$ such that $\overline{M}$ is a manifold. We lift the metric $ds^2$ onto $\overline{M}$, still denoted by $ds^2$. Then,
\begin{equation}\label{eqa-flat}
\Ric_{(\overline{M},ds^2)}(p)=0.
\end{equation}

Let $\alpha, \beta$ be two nontrivial simple closed curves on $S_g$ with the geometric intersection $i(\alpha, \beta)\geq 2$, and $\tau_{\alpha}, \tau_{\beta}$ be the Dehn-twists along $\alpha$ and $\beta$ respectively. Since the fundamental group $\pi_1(\overline{M})$ is a subgroup of $\Mod(S_g)$ of finite index, there exists $n_0, m_0 \in \mathbb{Z}^+$ such that  $\tau_{\alpha}^{n_0}, \tau_{\beta}^{m_0}\in \pi_1(\overline{M}).$ From Lemma \ref{dt-f} we know that the group 
\begin{equation}\label{free}
<\tau_{\alpha}^{n_0}, \tau_{\beta}^{m_0}>\cong \mathbb{F}_2
\end{equation}
where $\mathbb{F}_2$ is a free group of rank 2. For sure $<\tau_{\alpha}^{n_0}, \tau_{\beta}^{m_0}>$ acts on the universal cover $(\Teich(S_g), ds^2)$ of $(\overline{M},ds^2)$ by isometries.

We endow $<\tau_{\alpha}^{n_0}, \tau_{\beta}^{m_0}>$ with the word metric $dist_{word}$ w.r.t the generator set $\{\tau_{\alpha}^{n_0},\tau_{\alpha}^{-n_0}, \tau_{\beta}^{m_0}, \tau_{\beta}^{-m_0}\}$. Let $e$ be the unit in $<\tau_{\alpha}^{n_0}, \tau_{\beta}^{m_0}>$. For any $r>0$ we set
\begin{eqnarray*}
B(e, r):=\{\phi \in <\tau_{\alpha}^{n_0}, \tau_{\beta}^{m_0}>: \quad dist_{word}(\phi,e)\leq r\}.
\end{eqnarray*}

Let $q_0 \in  (\Teich(S_g), ds^2)$ and $dist_{ds^2}$ be the induced path metric of $ (\Teich(S_g), ds^2)$ on $\Teich(S_g)$. We define
$$C:=\max{\{dist_{ds^2}(\tau_{\alpha}^{n_0}\circ q_0, q_0), dist_{ds^2}(\tau_{\beta}^{m_0}\circ q_0, q_0)\}}>0.$$

The triangle inequality leads to
\begin{eqnarray}\label{3-1}
dist_{ds^2}(\phi \circ q_0, q_0)\leq r\cdot C, \quad \forall \phi \in B(e, r).
\end{eqnarray}

Since $<\tau_{\alpha}^{n_0}, \tau_{\beta}^{m_0}>$ acts freely on the universal cover $ (\Teich(S_g), ds^2)$ of $(\overline{M},ds^2)$, there exists a number $\epsilon_0>0$  such that 
$$dist_{ds^2}(\gamma \circ q_0, q_0)>2\epsilon_0, \quad \forall e\neq \gamma \in  <\tau_{\alpha}^{n_0}, \tau_{\beta}^{m_0}>$$
which implies 
\begin{eqnarray}\label{3-2}
\gamma_{1}\circ B(q_0;\epsilon_0)\cap \gamma_{2}\circ B(q_0;\epsilon_0)=\emptyset, \quad \forall \gamma_{1} \neq \gamma_{2} \in <\tau_{\alpha}^{n_0}, \tau_{\beta}^{m_0}>
\end{eqnarray}
where $B(q_0;\epsilon_0):=\{p\in(\Teich(S_g), ds^2); \ dist_{ds^2}(p,q_0)\leq \epsilon_0\}$.

From inequality (\ref{3-1}) and the triangle inequality we know that, for all $r>0$,
\begin{eqnarray}\label{3-3}
\bigcup_{\gamma \in B(e,r)}\gamma \circ B(q_0;\epsilon_0)\subset B(q_0;r\cdot C+\epsilon_0).
\end{eqnarray}

Equation (\ref{3-2}) tells that the geodesic balls $\{\gamma \circ B(q_0;\epsilon_0)\}_{\gamma \in B(e,r)}$ are pairwisely disjoint. Thus, by taking the volume, equations (\ref{3-2}) and (\ref{3-3}) lead to

\begin{eqnarray}\label{3-4}
\sum_{\gamma \in B(e,r)}\Vol(\gamma \circ B(q_0,\epsilon_0)) &=&  \Vol(\bigcup_{\gamma \in B(e,r)}\gamma \circ B(q_0;\epsilon_0)) \nonumber \\
&\leq& \Vol(B(q_0; r\cdot C+\epsilon_0)).
\end{eqnarray}

Since the group $<\tau_{\alpha}^{n_0}, \tau_{\beta}^{m_0}>$ acts on $ (\Teich(S_g), ds^2)$ by isometries, $\Vol(\gamma \circ B(q_0,\epsilon_0))=\Vol(B(q_0; \epsilon_0))$ for all $\gamma \in <\tau_{\alpha}^{n_0}, \tau_{\beta}^{m_0}>$. From inequality (\ref{3-4}) we have 

\begin{eqnarray}\label{3-5}
\#B(e,r)\cdot \Vol(B(q_0,\epsilon_0))\leq \Vol(B(q_0,r\cdot C+\epsilon_0)).
\end{eqnarray}

Rewrite it as
\begin{eqnarray}\label{3-6}
\#B(e, r)\leq \frac{\Vol(B(q_0;r\cdot C+\epsilon_0))}{\Vol(B(q_0;\epsilon_0))}.
\end{eqnarray}

Since $(\Teich(S_g),ds^2)$ is complete, from equation (\ref{eqa-flat}) and the Gromov-Bishop volume comparison inequality (see \cite{Gromovbook}), we have, for all $r>0$,
\begin{eqnarray}
\#B(e, r)&\leq &\frac{\Vol(B(q_0; r\cdot C+\epsilon_0))}{\Vol(B(q_0;\epsilon_0))}\\
&\leq& \frac{(r\cdot C+\epsilon_0)^{6g-6}}{\epsilon_0^{6g-6}}.
\end{eqnarray}
Which in particular implies that the group $<\tau_{\alpha}^{n_0}, \tau_{\beta}^{m_0}>\subset \Mod(S_g)$ has polynomial growth, which contradicts equation (\ref{free}) since the free group $\mathbb{F}_2$ has exponential growth.
\end{proof}

Let $M$ be a finite cover of $\mathbb{M}_g$ which is a manifold and $ds^2$ be a complete Riemannian metric on $M$ which has nonnegative scalar curvature. Since the metric is smooth, for any $p_0 \in M$ there exists a constant $r_1>0$ such that the geodesic ball $B(p_0;r_1)$ centered at $p_0$ of radius $r_1$ has smooth boundary $\partial B(p_0;r_1)$ and smooth outer normal derivative $\frac{\partial}{\partial \nu}$ on $\partial B(p_0,r_1)$. It suffices to choose $r_1$ to be less than the injectivity radius of $M$ at $p_0$. 

We let $\Sca_{ds^2}$ be the scalar curvature of $(M,ds^2)$ and $\Delta_{ds^2}$ be the Laplace operator of $(M,ds^2)$. Consider the operator
\begin{equation}\label{opera}
\L_{ds^2} (u):=-\frac{2(6g-7)}{3g-4} \Delta_{ds^2} u + \Sca_{ds^2} \cdot u
\end{equation}
where $u \in C^{\infty}((M,ds^2), \mathbb{R}).$

For $0<r<\inj(p_0)$ where $\inj(p_0)$ is the injectivity radius of $(M,ds^2)$ at $p_0$. Let $\mu_1(B(p_0;r))$ be the lowest eigenvalue of $\L$ with Neumann boundary conditions $\frac{\partial u}{\partial \nu}=0$ on $\partial B(p_0;r)$. It is well-known that 
\begin{equation}\label{fev}
\mu_1(B(p_0;r))=\inf_{v\in  C^{\infty}((M,ds^2), \mathbb{R})}\frac{\int_{B(p_0;r)}(||\grad v||_{ds^2}^2+ \Sca_{ds^2} \cdot v^2 )\dVol}{\int_{B(p_0;r)}v^2 \dVol}.
\end{equation}

The following result is Theorem A in \cite{Kazdan82} which is crucial in this section.  
\begin{theorem}[Kazdan]\label{Kazdan-4-1}
Assume that $\mu_1(B(p_0;r))>0$. Then there is a solution $u>0$ on $(M,ds^2)$ of $\L u>0$; in fact there exist two constants $C_1, C_2>0$ such that 
$$0<C_1\leq u(p) \leq C_2, \quad \forall p \in (M,ds^2).$$ 
\end{theorem}

Assume that $u>0$ on $(M,ds^2)$. We define the conformal metric
\begin{equation}\label{4-1}
ds_u^2:=u^{\frac{2}{3g-4}} \cdot ds^2
\end{equation}

Direct computation shows that the scalar curvature $\Sca_{ds_u^2}$ of $ds_u^2$ is given by the formula
\begin{eqnarray}\label{4-2}
\L_{ds^2}u &=&-\frac{2(6g-7)}{3g-4}\Delta_{ds^2} u+\Sca_{ds^2} \cdot u \\
&=& \Sca_{ds_u^2} \cdot u^{\frac{3g-2}{3g-4}}
\end{eqnarray}

Thus,
\begin{equation}\label{4-3}
\Sca_{ds_u^2}=\L_{ds^2}u \cdot u^{-\frac{3g-2}{3g-4}}
\end{equation}

\begin{proof}[Proof of Theorem \ref{deform}]
We follow exactly the same argument as in \cite{Kazdan82}. For the sake of completeness we sketch the proof.

First from Lemma \ref{norf} we know that that there exists a point $p_0\in M$ such that the Ricci tensor 
\begin{equation}\label{4-3-1}
\Ric_{(M,ds^2)}(p_0)\neq 0.
\end{equation} 

We let $r_1$ be a constant with $0<r_1<\inj(p_0)$. Pick a function $\eta \in C_0^{\infty}(B(p_0,r_1);\mathbb{R}^{\geq 0})$ with $\eta (p_0)>0$ and consider a family of metrics
$$ds_t^2:=ds^2-t\cdot \eta \cdot \Ric_{(M,ds^2)}$$
with scalar curvature $\Sca_{(M,ds_t^2)}$ and the corresponding operator $\L_{ds_t^2}$ defined in equation (\ref{opera}) with lowest Neumann eigenvalue $\mu_1(B(p_0,r_1),t)$. The first variation formula (see \cite{Kazdan82}) gives that

\begin{equation}\label{4-4}
\frac{d}{dt}\mu_{1}(B(p_0,r_1),t)|_{t=0}=\frac{\eta <\Ric, \Ric>}{\Vol((B(p_0,r_1))}
\end{equation}  
where $<.>$ is the standard inner product for tensors in the $ds^2$ metric.

Since $\eta(p_0)>0$, equations (\ref{4-3-1}) and (\ref{4-4}) give that 
\begin{equation}\label{4-4-1}
\frac{d}{dt}\mu_{1}(B(p_0,r_1),t)|_{t=0}>0.
\end{equation}  

Since $\Sca_{(M,ds^2)}\geq 0$, equation (\ref{fev}) gives that
$$\mu_1(B(p_0,r_1),0)=\mu_1(B(p_0,r_1)) \geq 0.$$

Thus, from inequality (\ref{4-4-1}) we know that for small enough $t_0>0$, 
\begin{equation}\label{4-5}
\mu_{1}(B(p_0,r_1),t_0)>0
\end{equation} 

It is clear that 
\begin{equation}\label{4-6}
ds_{t_0}^2 \asymp ds^2.
\end{equation}

Because of inequality (\ref{4-5}) we apply Theorem \ref{Kazdan-4-1} to $(M,ds_{t_0}^2)$. Thus, there is a smooth function $u$ on  $(M,ds_{t_0}^2)$ such that
\begin{equation}\label{4-7}
\L_{ds_{t_0}^2} u (p)>0 \quad \textit{and} \quad  u(p)>0,  \quad \forall p \in (M,ds_{t_0}^2).
\end{equation}
And there exist two constants $C_1, C_2>0$ such that 
\begin{equation}\label{4-8}
0<C_1\leq u(p) \leq C_2, \quad \forall p \in (M,ds_{t_0}^2).
\end{equation}

Then we define the new metric as 
\begin{equation}\label{4-9}
ds_1^2:=  u^{\frac{2}{3g-4}} \cdot ds_{t_0}^2.
\end{equation}

It is clear that Part (1) follows from equations (\ref{4-3}) and (\ref{4-7}). And Part (2) follows from equations (\ref{4-6}), (\ref{4-9}) and inequality (\ref{4-8}).
\end{proof}

\section{Proof of Theorem \ref{mt-1}}\label{proof of 1}
Before we prove Theorem \ref{mt-1}, let us make some preparation and fix the notations. 

Let $(M_1,ds_1^2), (M_2,ds_2^2)$ be two Finsler manifolds of the same dimensions, and $f:(M_1,ds_1^2) \to (M_2,ds_2^2)$ be a smooth map. For $C>0$, we call that  $f$ is an $C-contraction$ if for any $p \in M_1$, 
\begin{equation}
||f_{*}(V)||_{ds_2^2}\leq C \cdot ||V||_{ds_1^2}, \quad \forall V \in T_{p}M_1.
\end{equation}

Recall that the $degree$ $\deg(f)$ of $f$ is defined as

\begin{equation}\label{degree}
\deg(f)=\sum_{q \in f^{-1}(p)} \sign(\det f_{*}(q))
\end{equation}
where $p$ is a regular value of $f$. 

\begin{definition}
We call an $n$-dimensional Riemannian manifold $X$ is \textsl{hyperspherical} if for every $\epsilon>0$, there exists an $\epsilon$-contraction map $f_{\epsilon}:X \rightarrow \mathbb{S}^n$ of nonzero degree onto the standard unit $n$-sphere such that $f_{\epsilon}$ is a constant outside a compact subset in $X$.
\end{definition}

The following result was proved by Gromov and Lawson in \cite{GL83}.
\begin{theorem}[Gromov-Lawson]\label{gromovlawson}
A complete aspherical Riemannian manifold $X$ cannot have positive scalar curvature if the universal cover $\tilde{X}$ of $X$ is hyperspherical.
\end{theorem}

\begin{proof}
See the proof of Theorem 6.12 in \cite{GL83}.
\end{proof}

The classical Cartan-Hardamard theorem implies that any complete simply connected Riemannian manifold of nonpositive sectional curvature is hyperspherical. A direct corollary of Theorem \ref{gromovlawson} is 
\begin{theorem}[Gromov-Lawson]\label{gl-3}
Let $(X,ds_1^2)$ be a complete Riemannian manifold of nonpositive sectional curvature. Then for any Riemannian metric $ds_2^2$ on $X$ with $ds_2^2 \succ ds_1^2$, the scalar curvature of $(X,ds_2^2)$ cannot be positive everywhere on $X$.
\end{theorem}

\begin{proof}
One may see \cite{GL83} or \cite{Roe93} for the details.
\end{proof}

Now let us state the theorem (Theorem \ref{mt-1}) we will prove in this section.
\begin{theorem}\label{mt-1-1}
Let $S_{g}$ be a closed surface of genus $g$ with $g\geq 2$ and $M$ be a finite cover of the moduli space $\mathbb{M}_{g}$ of $S_{g}$. Then for any Riemannian metric $ds^2$ on $M$ with $ds^2 \succ ds_T^2$, we have
$$\inf_{p\in (M,ds^2)}\Sca(p)<0.$$
\end{theorem}

Recall that in the proof of Theorem \ref{gl-3} in \cite{GL83}, the nonpositivity of the sectional curvature is crucial because in this case the inverse of the exponential map is a contraction. In the setting of Theorem \ref{mt-1-1}, although the inverse of the exponential map is well-defined by Theorem \ref{teich}, it is far from a contraction. In fact, Masur in \cite{Masur75} showed that there exists two different geodesics in $(\Teich(S_g), ds_T^2)$ starting from the same point such that they have bounded Hausdorff distance. In particular, $(\Teich(S_g), ds_T^2)$ is not nonpositively curved in the sense of metric spaces. Hence, the argument in \cite{GL83} can not be directly applied to show Theorem \ref{mt-1-1}. Furthermore the following question is \textsl{unknown} as far as we know.

\begin{question}\label{ques}
Is $(\Teich(S_g), ds_T^2)$ hyperspherical? 
\end{question}

We are going to steer clear of Question \ref{ques} to prove Theorem \ref{mt-1-1}. It is very \textsl{interesting} to know the answer to Question \ref{ques}. We will discuss it from different viewpoints in the last section. \\

Before we prove Theorem \ref{mt-1-1}, we first provide two important properties for the Teichm\"uller space $(\Teich(S_g),ds^2)$ where $ds^2 \asymp ds_T^2$, which will be applied later.

\begin{definition}
Let $X$ be a metric space. We call $X$ is \textsl{uniformly contractible} if there is a function $f:(0,\infty) \to (0,\infty)$ so that for each $x\in X$ and $r>0$, the ball $B(x;r)$ of radius $r$ centered at $x$ is contractible in the concentric ball $B(x;f(r))$ of radius $f(r)$.
\end{definition}

\begin{proposition}\label{ucft}
Let $ds^2$ be a Riemannian metric on $\Teich(S_g)$ such that $ds^2 \asymp ds_T^2$. Then $(\Teich(S_g), ds^2)$ is uniformly contractible. 

In particular, the \tec space endowed with the perturbed Ricci metric $(\Teich(S_g),d_{LSY}^2)$ is uniformly contractible.
\end{proposition}

\begin{proof}
From Theorem \ref{lsy-04} we know that $ds_{LSY}^2 \asymp ds_T^2$. It suffices to show $(\Teich(S_g),ds^2)$ is uniformly contractible provided that $ds^2 \asymp ds_T^2$.

Since $ds^2 \asymp ds_T^2$, there exist two constants $k_1, k_2>0$ such that
$$k_1\cdot  ds_T^2 \leq ds^2 \leq k_2 \cdot ds_T^2.$$

In particular, we have, for each $p \in \Teich(S_g)$ and $r>0$
\begin{equation}\label{3-1-1}
B_{ds^2}(p;r)\subset B_{ds_T^2}(p;\frac{r}{k_1})\subset B_{ds^2}(p;\frac{k_2}{k_1}\cdot r)
\end{equation}
where $B_{ds^2}(p;r):=\{q\in \Teich(S_g); \ \dist_{ds^2}(p,q)\leq r\}$ and$B_{ds_T^2}(p;r):=\{q\in \Teich(S_g); \ \dist_{ds_T^2}(p,q)\leq r\}$.

Proposition \ref{contract} tells that the Teichm\"uller ball $B_{ds_T^2}(p;r)$ is contractible for all $r>0$ and $p\in \Teich(S_g)$. Thus, equation (\ref{3-1-1}) tells that $B_{ds^2}(p;r)$ is contractible in $B_{ds^2}(p;\frac{k_2}{k_1}\cdot r)$. Therefore, the conclusion follows by choosing 
$$f(r)=\frac{k_2}{k_1}\cdot r.$$
\end{proof}

\begin{definition}\label{2}
Let $X$ be a metric space. We call that $X$ has \textsl{bounded geometry} in the sense of coarse geometry if for every $\epsilon>0$ and every $r>0$, there exists an integer $n(r,\epsilon)>0$ such that for each $x\in X$ every ball $B(x;r)$ contains at most $n(r,\epsilon)$ $\epsilon$-disjoint points. Where $\epsilon$-disjoint means that any two different points are at at least $\epsilon$ distance from each other.
\end{definition}

\begin{proposition}\label{bgft}
$(\Teich(S_g),d_{LSY}^2)$ have bounded geometry in the sense of coarse geometry.
\end{proposition}
\begin{proof}
Let $\epsilon_0>0$ be the constant which is the lower bound for the injectivity radius of $(\Teich(S_g), ds_{LSY}^2)$ in Theorem \ref{lsy-04}. For every $r>0$ and every $\epsilon>0$. Let $p\in \Teich(S_g)$ and $B(p;r):=\{q\in \Teich(S_g); \ \dist_{ds_{LSY}^2}(p,q)\leq r\}$ be the geodesic ball of radius $r$ centered at p.  Assume $K=\{x_i\}_{i=1}^{k}$ be an arbitrary $\epsilon$- disjoint points in $B_{ds_{LSY}^2}(p;r)$. That is 
\begin{equation}\label{3-2-0}
\dist_{ds_{LSY}^2}(x_i,x_j)\geq \epsilon, \quad \forall 1\leq i\neq j\leq k.
\end{equation}

Let $\epsilon_1=\min{\{\frac{\epsilon}{4}, \epsilon_0\}}$. First the triangle inequality tells that
\begin{equation}\label{3-2-1}
\cup_{i=1}^{k}B(x_i; \epsilon_1) \subset B(p,r+\epsilon_1)\subset  B(p,r+\epsilon_0).
\end{equation}

By our assumptions that $\epsilon_1 \leq \frac{\epsilon}{4}$, inequality (\ref{3-2-0}) gives that
\begin{equation}\label{3-2-2}
B(x_i; \epsilon_1) \cap B(x_j; \epsilon_1) =\emptyset,  \quad \forall 1\leq i\neq j\leq k .
\end{equation}

From equations (\ref{3-2-1}) and (\ref{3-2-2}) we have
\begin{equation}\label{3-2-3}
\sum_{i=1}^k \Vol(B(x_i; \epsilon_1))\leq \Vol(B(p,r+\epsilon_0)).  
\end{equation}

From Theorem \ref{lsy-04} there exists a lower bound for the sectional curvatures of $(\Teich(S_g), ds_{LSY}^2)$, by using the Gromov-Bishop volume comparison inequality, in particular we have that there exists a constant $C(r, \epsilon_0, g)>0$ depending on $r, \epsilon_0$ and the genus $g$ such that the volume
\begin{equation}\label{3-2-4}
\Vol( B(p,r+\epsilon_0))\leq C(r, \epsilon_0, g).
\end{equation}

On the other hand, from Theorem \ref{lsy-04} we know that the sectional curvatures of $(\Teich(S_g), ds_{LSY}^2)$ have a upper bound. Since 
$$\epsilon_1 \leq \epsilon_0=\inj(\Teich(S_g), ds_{LSY}^2),$$ 

Elementary Riemannian geometry tells that there exists a constant $D(\epsilon_1, g)>0$ depending on $\epsilon_1$ and the genus $g$ such that the volume
\begin{equation}\label{3-2-5}
\Vol( B(x_i,\epsilon_1))\geq D(\epsilon_1, g)>0, \quad \forall 1\leq i \leq k.
\end{equation}

Inequalities (\ref{3-2-3}), (\ref{3-2-4}) and (\ref{3-2-5}) give that 
\begin{equation}\label{3-2-6}
k\leq \frac{C(r, \epsilon_0, g)}{ D(\epsilon_1, g)}.
\end{equation}
Then the conclusion follows by choosing 

$$n(r,\epsilon)=\frac{C(r, \epsilon_0, g)}{ D(\epsilon_1, g)}.$$
\end{proof}

The following result of A. N. Dranishnikov in \cite{Drani03} will be applied to prove Theorem \ref{mt-1-1}. 
\begin{theorem}[Dranishnikov]\label{shyper}
Let $X$ be a complete uniformly contractible Riemannian manifold with bounded geometry whose asymptotic dimension is finite, then the product $X\times \mathbb{R}^n$, endowed with the product metric, is hyperspherical for some positive number $n \in \mathbb{Z}$.
\end{theorem}

We remark here that the statement of the theorem above is different from Theorem 5 (or Theorem B) in \cite{Drani03} where there is no condition on the bounded geometry. But if one checks the proof of Theorem 5 in \cite{Drani03}, Theorem 5 follows from Theorem 4 and Lemma 4 in \cite{Drani03} where Theorem 4 requires that the space $X$ has bounded geometry. We are grateful to Prof. Dranishnikov for the clarification.

Now we are ready to prove Theorem \ref{mt-1}.
\begin{proof}[Proof of Theorem \ref{mt-1}]
Let $M$ be a finite cover of the moduli space $\mathbb{M}_g$ of $S_g$ and $ds^2$ be a Riemannian metric on $M$ such that $ds^2 \succ ds_T^2$. That is, there exists a constant $k_1>0$ such that 
\begin{equation}\label{1-1-1}
ds^2 \geq k_1 \cdot d_T^2.
\end{equation}
We argue it by contradiction. Assume that 
\begin{equation}\label{1-1-2}
\inf_{p\in (M,ds^2)}\Sca(p)\geq 0.
\end{equation}

If necessary, we pass to a finite cover of $M$, still denoted by $M$, such that $M$ is a manifold. From inequality (\ref{1-1-1}) we know that $(M,ds^2)$ is complete since the Teichm\"uller metric is complete. Thus, from Theorem \ref{deform} we know that there exists a new metric $ds_1^2$ on $M$ such that
\begin{equation}\label{1-1-3}
\Sca(p)> 0, \ \forall p\in (M,ds_1^2)
\end{equation}

and 
\begin{equation}\label{1-1-4}
ds_1^2\asymp ds^2.
\end{equation}

Let $ds_{LSY}^2$ be the perturbed Ricci metric on $M$. In fact either the McMullen metric or the Ricci metric also works here. From Proposition \ref{ucft}, Proposition \ref{bgft} and Theorem \ref{finite} we know that the universal cover $(\Teich(S_g), ds_{LSY}^2)$ of $(M,ds_{LSY}^2)$ is uniformly contractible, has bounded geometry and
$$\asydim((\Teich(S_g),ds_{LSY}^2))<\infty.$$

Hence, one may apply Theorem \ref{shyper} to get a positive integer $n$ such that $(\Teich(S_g), ds_{LSY}^2)\times \mathbb{R}^n$, endowed with the product metric, is hyperspherical. 

We pick this integer $n\in \mathbb{Z}^+$ and consider the product space 
$$(\Teich(S_g),ds_1^2)\times \mathbb{R}^n$$ 
where $(\Teich(S_g),ds_1^2)$ is the universal cover of $(M,ds_1^2)$.

It is clear that $(\Teich(S_g),ds_1^2)\times \mathbb{R}^n$ is a complete $(6g-6+n)$-dimensional Riemannian manifold, and the scalar curvature of $ (\Teich(S_g),ds_1^2)\times \mathbb{R}^n$ satisfies that
\begin{equation}\label{1-1-5}
 \Sca((p,v))=\Sca(p)> 0
\end{equation}
where $(p,v)$ is arbitrary in $(\Teich(S_g),ds_1^2)\times \mathbb{R}^n$.\\

\textsl{Claim: The complete product manifold $(\Teich(S_g),ds_1^2)\times \mathbb{R}^n$ is hyperspherical.}\\

\begin{proof}[Proof of the Claim:]

First since $ds_1^2 \asymp ds^2$ (see equation (\ref{1-1-4})), the identity map
\begin{equation}\label{1-1-6}
i_1: (\Teich(S_g), ds_1^2) \times \mathbb{R}^n \to (\Teich(S_g), ds^2) \times \mathbb{R}^n .
\end{equation}
is a $c_1$-contraction diffeomorphism for some constant $c_1\geq 1$.

Since we assume that $ds^2 \succ ds_T^2$ (by assumption), the identity map
\begin{equation}\label{1-1-7}
i_2: (\Teich(S_g), ds^2) \times \mathbb{R}^n  \to (\Teich(S_g), ds_T^2) \times \mathbb{R}^n .
\end{equation}
is a $c_2$-contraction diffeomorphism for some constant $c_2\geq 1$.

From Theorem \ref{lsy-04} we know that $ds_T^2 \asymp ds_{LSY}^2$. Thus, the identity map
\begin{equation}\label{1-1-8}
i_3: (\Teich(S_g), ds_T^2)  \times \mathbb{R}^n \to (\Teich(S_g), ds_{LSY}^2) \times \mathbb{R}^n .
\end{equation}
is a $c_3$-contraction diffeomorphism for some constant $c_3\geq 1$.

By the choice of $n\in \mathbb{Z}^+$ we know that for every $\epsilon>0$ there exists an $\epsilon$-contraction map
\begin{equation}\label{1-1-9}
f_\epsilon: (\Teich(S_g), ds_{LSY}^2)  \times \mathbb{R}^n \to  \mathbb{S}^{6g-6+n}.
\end{equation}
such that $f_\epsilon$ is of nonzero degree onto the unit $(6g-6+n)$-sphere and $f_\epsilon$ is a constant outside a compact subset in $(\Teich(S_g), ds_{LSY}^2)  \times \mathbb{R}^n$.

Consider the following composition map
\begin{eqnarray}\label{1-1-10}
F_\epsilon: (\Teich(S_g), ds_1^2)  \times \mathbb{R}^n &\to&  \mathbb{S}^{6g-6+n} \\
(p,v)&\mapsto& f_\epsilon \circ i_3 \circ i_2 \circ i_1 (p,v) \nonumber
\end{eqnarray}
where  $(p,v)$ is arbitrary in $(\Teich(S_g), ds_1^2)  \times \mathbb{R}^n$.

Since $i_1, i_2$ and $ i_3$ are diffeomorphisms and $f_\epsilon$ has nonzero degree, $F_\epsilon$ also has nonzero degree by the definition.

Since $i_1, i_2$ and $ i_3$ are diffeomorphisms and $f_\epsilon$ is a constant outside a compact subset of $(\Teich(S_g), ds_{LSY}^2)  \times \mathbb{R}^n$, a standard argument in set-point topology gives that $F_\epsilon$ is also a constant outside a compact subset of $(\Teich(S_g), ds_{LSY}^2)  \times \mathbb{R}^n$.

It is clear that $F_\epsilon$ is onto because $i_1, i_2, i_3$ and $f_\epsilon$ are onto.

It remains to show that $F_\epsilon$ is a contraction. For every point $(p,v)\in (\Teich(S_g), ds_1^2)  \times \mathbb{R}^n$ and any tangent vector $W\in T_{(p,v)}((\Teich(S_g), ds_1^2)  \times \mathbb{R}^n)=\mathbb{R}^{6g-6+n}$,

\begin{eqnarray*}\label{1-1-11}
||(F_{\epsilon})_{*}(W)||&=&|| (f_\epsilon \circ i_3 \circ i_2 \circ i_1)_{*}(W)||\\
&\leq & \epsilon \cdot|| ( i_3 \circ i_2 \circ i_1)_{*}(W) ||\\
&\leq & \epsilon \cdot c_3 \cdot || ( i_2 \circ i_1)_{*}(W) ||\\
&\leq & \epsilon \cdot c_3 \cdot c_2 \cdot || ( i_1)_{*}(W) ||\\
&\leq &  \epsilon \cdot c_3 \cdot c_2 \cdot c_1 \cdot ||W||
\end{eqnarray*}
where $||\cdot||$ is the standard Euclidean norm in $\mathbb{R}^{6g-6+n}$.

Since $\epsilon>0$ is arbitrary and $c_1, c_2, c_3>0$, the claim follows.

\end{proof}

From the claim above and Theorem \ref{gromovlawson}  of Gromov-Lawson we know that the product manifold $(\Teich(S_g),ds_1^2)\times \mathbb{R}^n$ cannot have positive scalar curvature which contradicts inequality (\ref{1-1-5}).

\end{proof}

\begin{remark}
The following more general statement follows from exactly the same argument as the proof of Theorem \ref{mt-1}.
\begin{theorem}
Let $S_{g}$ be a closed surface of genus $g$ with $g\geq 2$ and $M$ be any cover which may be an infinite cover,  of the moduli space $\mathbb{M}_{g}$ of $S_{g}$ such that the orbiford fundamental group of $M$ contains a free subgroup of rank $\geq 2$. Then for any Riemannian metric $ds^2$ on $M$ with $ds^2 \succ ds_T^2$ we have
$$\inf_{p\in (M,ds^2)}\Sca(p)<0.$$
\end{theorem}
\end{remark}

\section{Proof of Theorem \ref{mt-2}}\label{proof of 2}
We start with the following definition.
\begin{definition}
Let $M$ be a cover, which may be an infinite cover, of the moduli space $\mathbb{M}_g$ and $ds^2$ be a Riemannian metric on $M$. We call that $ds^2$ is \textsl{quasi-isometric} to the Teichm\"uller metric $ds_T^2$ if there exist two positive constants $L\geq 1$ and $K\geq 0$ such that on the universal cover $(\Teich(S_g), ds^2)$ ($(\Teich(S_g), ds_T^2)$) of  $(M, ds^2)$ ($(M, ds_T^2)$) respectively, the identity map satisfies
$$ L^{-1}\dist_{ds_T^2}(p,q)-K \leq \dist_{ds^2}(p,q)\leq L \,\dist_{ds_T^2}(p,q)+K, \quad \forall p, q \in \Teich(S_g).$$
\end{definition}

If $K=0$, $ds^2$ is equivalent to $ds_T^2$.

In the quasi-isometry setting, the identity map, defined in the equation (\ref{1-1-7}) in the proof of Theorem \ref{mt-1}, may not be a contraction. Therefore, the proof of Theorem \ref{mt-1} can not directly lead to Theorem \ref{mt-2}. Instead of applying Theorem \ref{gromovlawson} in the proof of Theorem \ref{mt-1}, we apply the following theorem of Yu in \cite{Yu98} to prove Theorem \ref{mt-2}. One can see Corollary 7.3 in \cite{Yu98}.

\begin{theorem}[Yu]\label{yu}
A uniformly contractible Riemannian manifold with finite asymptotic dimension cannot have uniform positive scalar curvature.
\end{theorem}

Now we are ready to prove Theorem \ref{mt-2}.

\begin{proof}[Proof of Theorem \ref{mt-2}]
Let $M$ be a cover of the moduli space $\mathbb{M}_g$ of $S_g$ and $ds^2$ be a Riemannian metric on $M$ such that $ds^2$ is quasi-isometric to $ds_T^2$. That is, there exist two constants $L\geq1$ and $K>0$ such that 
\begin{equation}\label{2-1-1}
L^{-1}\dist_{ds_T^2}(p,q)-K \leq \dist_{ds^2}(p,q)\leq L \, \dist_{ds_T^2}(p,q)+K, \quad \forall p, q \in \Teich(S_g).
\end{equation}

Since the asymptotic dimension is a quasi-isometric invariance, Theorem \ref{asyofteich} gives that 
\begin{equation}
\asydim((\Teich(S_g), ds^2))<\infty.
\end{equation}

From Theorem \ref{yu} of Yu, it remains to show that $(\Teich(S_g), ds^2)$ is uniformly contractible. We follow a similar argument in the proof of Proposition \ref{ucft} to finish the proof.

For each $p \in \Teich(S_g)$ and every $r>0$, inequality (\ref{2-1-1}) and the triangle inequality lead to
\begin{equation}\label{6-1-1}
B_{ds^2}(p;r)\subset B_{ds_T^2}(p;L\cdot (r+K))\subset B_{ds^2}(p; L^2\cdot (r+K)+K)
\end{equation}
where $B_{ds^2}(p;r):=\{q\in \Teich(S_g); \ \dist_{ds^2}(p,q)\leq r\}$ and$B_{ds_T^2}(p;r):=\{q\in \Teich(S_g); \ \dist_{ds_T^2}(p,q)\leq r\}$.

Proposition \ref{contract} tells that the Teichm\"uller ball $B_{ds_T^2}(p;L\cdot(r+K))$ is contractible for all $r>0$ and $p\in \Teich(S_g)$. Thus, equation (\ref{6-1-1}) tells that $B_{ds^2}(p;r)$ is contractible in $B_{ds^2}(p; L^2\cdot (r+K)+K)$. Thus, the conclusion follows by choosing 
$$f(r)= L^2\cdot (r+K)+K.$$
\end{proof}

\begin{remark}
Theorem \ref{mt-2} also holds in the following sense of quasi-isometry, where we call that $ds^2$ is \textsl{quasi-isometric} to the Teichm\"uller metric $ds_T^2$ if there exist two positive constants $L\geq 1$, $K\geq 0$ and a map $$f: (\Teich(S_g), ds_T^2) \to (\Teich(S_g), ds^2)$$ such that for all $p, q \in \Teich(S_g)$,
$$ L^{-1}\dist_{ds_T^2}(p,q)-K \leq \dist_{ds^2}(f(p),f(q))\leq L \,\dist_{ds_T^2}(p,q)+K.$$
If we assume that $(\Teich(S_g), ds^2)$ is quasi-isometric to $(\Teich(S_g), ds_T^2)$, then the space $(\Teich(S_g), ds^2)$ is also quasi-isometric to $(\Teich(S_g), ds_M^2)$ or $(\Teich(S_g), ds_{LSY}^2)$, where these two metrics are uniformly contractible. Indeed, Theorem 7.1 in \cite{Yu98} and Theorem \ref{finite} give that the coarse Baum-Connes conjecture holds for $(\Teich(S_g), ds_M^2)$ or $(\Teich(S_g), d_{LSY}^2)$. Then one can see Corollary 3.9 in \cite{Roe96} and use the same argument in the proof of Theorem \ref{mt-2} to get the conclusion. 
\end{remark}

\section{One question}\label{question}

Let $N$ be an $n$-dimensional complete simply-connected Riemannian manifold of nonpositive sectional curvature. For any $p\in N$ , let $T_pN$ be the tangent space of $N$ at $p$, it is well-known that the inverse of the exponential map at $p$
$$\exp_{p}^{-1}:N \to T_pN=\mathbb{R}^n$$
is a $1$-contraction diffeomorphism. And this property plays an important role in the proof of Theorem \ref{gl-3} of Gromov-Lawson.
The following question arises in this project.
\begin{question}\label{mq-1}
Is there any proper differential map 
$$F:(\Teich(S_g),ds_T^2) \to \mathbb{R}^{6g-6}$$ 
such that $F$ is a $1$-contraction of degree one? Moreover, could $F$ be a diffeomorphism?
\end{question}

The constant $1$ for the contraction property in this question is not essential, since one can take a rescaling on the target space $\mathbb{R}^{6g-6}$. An affirmative answer to Question \ref{mq-1} will give another proof of Theorem \ref{mt-1} by following exact the same argument in \cite{GL83}. See Theorem \ref{gromovlawson}.

We enclose this section by recalling several well-known parametrizations which may be helpful for this question.\\

$(1)$. Recall that the Teichm\"uller parametrization at a point $X \in \Teich(S_g)$ is given by
\begin{eqnarray*}
F_X:(\Teich(S_g),ds_T^2) &\to& \mathbb{R}^{6g-6}=\mathbb{S}^{6g-7}\times \mathbb{R}^{\geq 0}\\
                              Y &\mapsto& (V[X,Y],\dist_T(X,Y))
\end{eqnarray*}
where $V[X,Y]$ is the direction of the Teichm\"uller geodesic from $X$ to $Y$.

It is well-known that $F_X$ is a proper differential map of degree one. However, $F_X$ is not a contraction since there exists two geodesics starting at $X$ which have bounded Hausdorff distance, which was proved by Masur in \cite{Masur75}. \newline
 
$(2)$. Let $X \in \Teich(S_g)$ be a hyperbolic surface and $\QD(X)$ be the holomorphic quadratic differential on $X$ which can be identified with $\mathbb{R}^{6g-6}$. Let $\beta_{X}$ be the Bers embedding of $(\Teich(S_g),ds_T^2)$ into $\QD(X)$ with respect to the base point $X$.

It is well-known that $\beta_X$ is a contraction (For example one can see Theorem 4.3 in \cite{FKM13}). However, $\beta_X$ is not proper since the image of the Bers embedding is a bounded subset in $\mathbb{R}^{6g-6}$.\newline

$(3)$. Fix a hyperbolic surface $X \in \Teich(S_g)$. Then for any $Y \in \Teich(S_g)$ there exists a unique harmonic map from $X$ to $Y$ which is isotopic to the identity map. The Hopf differential $F_X(Y)$ of this harmonic map is a holomorphic quadratic differential on $X$. In particular this gives a differential map from $\Teich(S_g)$ to $\QD(X)$. Wolf in \cite{Wolf89} showed that this map
$$F_X: (\Teich(S_g),ds_T^2) \to \QD(X)=\mathbb{R}^{6g-6}$$
is a diffeomorphism. In particular, this map is proper of degree one. However,  Markovic in \cite{Mark02} showed that $F_X$ is not a contraction (One can see Theorem 2.2 in \cite{Mark02}). \newline

$(4)$. Since $ds_T^2 \succ ds_{WP}^2$, the identity map
$$i:(\Teich(S_g),ds_T^2) \to (\Teich(S_g),ds_{WP}^2)$$
is a contraction diffeomorphism. Fix a hyperbolic surface $X \in \Teich(S_g)$. Since the sectional curvature of the Weil-Petersson metric is negative \cite{Tromba86,Wolpert86} and the Weil-Petersson metric is geodescally convex \cite{Wolpert87},  the inverse of the exponential map at $X$
$$\exp_{X}^{-1}:\Teich(S_g) \to T_X(\Teich(S_g))=\mathbb{R}^{6g-6}$$
is a $1$-contraction. Consider the composition map
\begin{eqnarray*}
F_X: (\Teich(S_g),ds_T^2)&\to&  T_X(\Teich(S_g))=\mathbb{R}^{6g-6}\\
Y &\mapsto& \exp_X^{-1}\circ i (Y) 
\end{eqnarray*}
It is not hard to see that $F_X$ is a contraction and differential map of degree one. However, since the Weil-Petersson metric is incomplete \cite{Chu76, Wolpert75}, $F_X$ is not proper.

\bibliographystyle{amsalpha}
\bibliography{ref}

\end{document}